\documentclass[11pt]{amsart}
\usepackage{amssymb,amsfonts,graphicx,color,enumerate,calc}
\usepackage[lmargin=30mm,rmargin=30mm,tmargin=30mm,bmargin=30mm]{geometry}
\usepackage{float}
\renewcommand{\baselinestretch}{1.1}
\usepackage[numbers,sort&compress]{natbib}

\usepackage{mdwlist}
\usepackage[colorlinks=true,citecolor=black,linkcolor=black,urlcolor=blue]{hyperref}
\newcommand{\arXiv}[1]{arXiv:\,\href{http://arxiv.org/abs/#1}{#1}}

\theoremstyle{plain} 
\newtheorem{theorem}{Theorem}
\newtheorem{lemma}[theorem]{Lemma}
\newtheorem{corollary}[theorem]{Corollary}

\theoremstyle{definition} 
\newtheorem{conjecture}[theorem]{Conjecture}
\newcommand{\ceil}[1]{\ensuremath{\protect\lceil#1\rceil}}

\newcommand{\floor}[1]{\ensuremath{\protect\lfloor#1\rfloor}}

\newcommand{\PP}{\ensuremath{\mathcal{P}}}
\newcommand{\BB}{\ensuremath{\mathcal{B}}}
\newcommand{\had}{\textsf{\textup{had}}}
\newcommand{\glm}{\textsf{\textup{glm}}}
\newcommand{\tw}{\textsf{\textup{tw}}}

\newcommand{\bn}{\textsf{\textup{bn}}}
\newcommand{\tn}{\textsf{\textup{tn}}}
\newcommand{\bw}{\textsf{\textup{bw}}}
\newcommand{\sep}{\textsf{\textup{sep}}}
\newcommand{\sn}{\textsf{\textup{sn}}}
\newcommand{\ltp}{\textsf{\textup{ltp}}}
\newcommand{\ctp}{\textsf{\textup{ctp}}}
\newcommand{\link}{\textsf{\textup{link}}}
\newcommand{\wl}{\textsf{\textup{wl}}}

\begin{document} 

\title{Parameters Tied to Treewidth}
\author{Daniel J. Harvey}
\address{\newline School of Mathematical Sciences
\newline Monash University
\newline Melbourne, Australia}
\email{daniel.harvey@monash.edu}
\author{David~R.~Wood}
\address{\newline School of Mathematical Sciences
\newline Monash University
\newline Melbourne, Australia}
\email{david.wood@monash.edu}
\date{\today}
\thanks{Research of D.R.W. is supported by the Australian Research Council.}
\thanks{Research of D.J.H. is supported by an Australian Postgraduate Award.}

\begin{abstract}
Treewidth is a graph parameter of fundamental importance to algorithmic and structural graph theory. This paper surveys several graph parameters tied to treewidth, including separation number, tangle number, well-linked number and Cartesian tree product number. We review many results in the literature showing these parameters are tied to treewidth. In a number of cases we also improve known bounds, provide simpler proofs and show that the inequalities presented are tight.
\end{abstract}

\maketitle

\section{Introduction}
Treewidth is an important graph parameter for two key reasons. Firstly, treewidth has many algorithmic applications; for example, there are many results showing that NP-hard problems can be solved in polynomial time on classes of graphs with bounded treewidth (see \citet{Bodlaender-AC93} for a survey). Treewidth is inherently related to graph separators, which are ``small" sets of vertices whose removal leaves no component with more than half the vertices (or thereabouts). Separators are particularly useful when using dynamic programming to solve graph problems; find and delete a separator, recursively solve the problem on the remaining components, and then combine these solutions to obtain a solution for the original problem.

Secondly, treewidth is a key parameter in graph structure theory, especially in Robertson and Seymour's seminal series of papers on graph minors \citep{RS-GraphMinors}. Ultimately, the purpose of these papers was to prove what it now known as the Graph Minor Theorem (often referred to as Wagner's Conjecture), which states that any class of minor-closed graphs (other than the class of all graphs) has a finite set of forbidden minors. In order to prove this, Robertson and Seymour separately considered classes with bounded treewidth and classes with unbounded treewidth. The Graph Minor Theorem is (comparatively) easy to prove for classes with bounded treewidth \cite{RS-GraphMinorsIV-JCTB90}. In order to prove the Graph Minor Theorem for classes with unbounded treewidth, Robertson and Seymour showed that graphs with large treewidth contain large grid minors. This Grid Minor Theorem has been reproved by many researchers; we will discuss it more thoroughly in Section~\ref{section:gridminors}. In proving these results, the parameters \emph{linkedness} and \emph{well-linked number} were used. At the heart of the Graph Minor Theorem is the Graph Minor Structure Theorem, which describes how to construct a graph in a minor-closed class; see \citet{KM-GC07} for a survey of several versions of the Graph Minor Structure Theorem. The most complex version, and the one used in the proof of the Graph Minor Theorem, describes the structure of graphs in a minor-closed class with unbounded treewidth in terms of \emph{tangles}. Robertson and Seymour combined all these ingredients in their proof of the Graph Minor Theorem. 

The purpose of this paper is to survey a number of known graph parameters that are closely related to treewidth, including those mentioned above such as separation number, linkedness, well-linked number and tangle number. 

Formally, a \emph{graph parameter} is a real-valued function $\alpha$ defined on all graphs such that $\alpha(G_1)=\alpha(G_2)$ whenever $G_1$ and $G_2$ are isomorphic. Two graph parameters $\alpha(G)$ and $\beta(G)$ are \emph{tied}\footnotemark[2]\footnotetext[2]{Occasionally, other authors use the term \emph{comparable} \citep{Fox11}} if there exists a function $f$ such that for every graph $G$, $$\alpha(G) \leq f(\beta(G)) \text{ and } \beta(G) \leq f(\alpha(G)).$$ Moreover, say that $\alpha$ and $\beta$ are \emph{polynomially tied} if $f$ is a polynomial.

We survey results from the literature that together prove the following theorem:
\begin{theorem}
\label{theorem:pttt}
  The following graph parameters are polynomially tied:
  \begin{itemize*}
  \item treewidth,
  \item bramble number,
  \item minimum integer $k$ such that $G$ is a spanning subgraph of a $k$-tree,
  \item minimum integer $k$ such that $G$ is a spanning subgraph of a chordal graph with no $(k+2)$-clique,
  \item separation number,
  \item branchwidth,
  \item tangle number,
  \item lexicographic tree product number,
  \item Cartesian tree product number,
  \item linkedness,
  \item well-linked number,
  \item maximum order of a grid minor,
  \item maximum order of a grid-like-minor,
  \item Hadwiger number of the Cartesian product $G\square K_2$ (viewed as a function of $G$),
  \item fractional Hadwiger number,
  \item $r$-integral Hadwiger number for each $r\geq 2$.
  \end{itemize*}
\end{theorem}

\citet{Fox11} states (without proof) a theorem similar to Theorem~\ref{theorem:pttt} with the parameters treewidth, bramble number, separation number, maximum order of a grid minor, fractional Hadwiger number, and $r$-integral Hadwiger number for each $r\geq 2$. Indeed, this statement of Fox motivated the present paper.

This paper surveys the parameters in Theorem~\ref{theorem:pttt}, showing where these parameters have been useful, and provides proofs that each parameter is tied to treewidth (except in a few cases). In a number of cases we improve known bounds, provide simpler proofs and show that the inequalities presented are tight. The following graph is a key example. Say $n,k$ are integers. Let $\psi_{n,k}$ be the graph with vertex set $A \cup B$, where $A$ is a clique on $n$ vertices, $B$ is an independent set on $kn$ vertices, and $A \cap B = \emptyset$, such that each vertex of $A$ is adjacent to exactly $k(n-1)$ vertices of $B$ and each vertex of $B$ is adjacent to  exactly $n-1$ vertices of $A$. (Note it always possible to add edges in this fashion; pair up each vertex in $A$ with $k$ vertices in $B$ such that all pairs are disjoint, and then add all edges from $A$ to $B$ except those between paired vertices.)

\begin{figure}[H]
\centering
\includegraphics{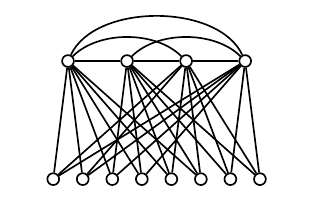}
\caption{The graph $\psi_{4,2}$.}
\end{figure}

\section{Treewidth and Basics}
\label{section:ptttbasics}
Let $G$ be a graph.  
A \emph{tree decomposition} of $G$ is a pair $(T, (B_x \subseteq V(G))_{x \in V(T)})$ consisting of 
\begin{itemize}
\item a tree $T$,
\item a collection of \emph{bags} $B_x$ containing vertices of $G$, indexed by the nodes of $T$.
\end{itemize}
The following conditions must also hold:
\begin{itemize}
\item For all $v \in V(G)$, the set $\{ x \in V(T) : v \in B_x\}$ induces a non-empty subtree of $T$.
\item For all $vw \in E(G)$, there is some bag $B_x$ containing both $v$ and $w$.
\end{itemize}
The \emph{width} of a tree decomposition is defined as the size of the largest bag minus 1. The treewidth $\tw(G)$ is the minimum width over all tree decompositions of $G$. Often, for the sake of simplicity, we will refer to a tree decomposition simply as $T$, leaving the set of bags implied whenever this is unambiguous. For similar reasons, often we say that bags $X$ and $Y$ are adjacent (or we refer to an edge $XY$), instead of the more accurate statement that the nodes of $T$ indexing $X$ and $Y$ are adjacent.

Treewidth was defined by \citet{Halin76} (in an equivalent form which Halin called $S$-functions) and independently by \citet{RS-GraphMinorsII-JAlg86}. Intuitively, a graph with low treewidth is simple and treelike---note that a tree itself has treewidth 1. (In fact, ensuring this fact is the reason that 1 was subtracted in the definition of width.) On the other hand, a complete graph $K_{n}$ has treewidth $n-1$.  
Say a tree decomposition is \emph{normalised} if each bag has the same size and $|X\backslash Y|=|Y\backslash X|=1$ whenever $XY$ is an edge.
\begin{lemma}
\label{lemma:normal}
If a graph $G$ has a tree decomposition of width $k$, then $G$ has a normalised tree decomposition of width $k$.
\end{lemma}
\begin{proof}
Let $T$ be a tree decomposition of $G$ with width $k$. Thus $T$ contains a bag of size $k+1$. If some bag of $T$ does not contain $k+1$ vertices, then since $T$ is connected, there exist adjacent bags $X$ and $Y$ such that $|X|=k+1$ and $|Y|<k+1$. Then $X \backslash Y$ is non-empty; take some vertex of $X \backslash Y$ and add it to $Y$. Repeat this process until all bags have size $k+1$.

Now, consider an edge $XY$. Since $|X|=|Y|$, it follows $|X \backslash Y|=|Y \backslash X|$. If $|X \backslash Y| > 1$, then let $v \in X \backslash Y$ and $u \in Y \backslash X$. Subdivide the edge $XY$ of $T$ and call the new bag $Z$. Let $Z = (X \backslash \{v\}) \cup \{u\}$. Now $|X \backslash Z| = 1$ and $|Y \backslash Z|=|Y \backslash X|-1$. Repeat this process until $|X \backslash Y|=|Y \backslash X| \leq 1$ for each pair of adjacent bags. Finally, if $XY$ is an edge and $|X \backslash Y|=0$, then contract the edge $XY$, and let the bag at the contracted node be $X$. Repeat this process so that if $X$ and $Y$ are a pair of adjacent bags, then $|X \backslash Y|=|Y \backslash X|=1$. All of these operations preserve tree decomposition properties and width. Hence this modified $T$ is our desired normalised tree decomposition.
\end{proof}

A \emph{$k$-colouring} of a graph $G$ is a function that assigns one of $k$ colours to each vertex of $G$ such that no pair of adjacent vertices are assigned the same colour. The \emph{chromatic number} $\chi(G)$ is the minimum number $k$ such that $G$ has a $k$-colouring.

A graph $H$ is a \emph{minor} of a graph $G$ if a graph isomorphic to $H$ can be constructed from $G$ by vertex deletion, edge deletion and edge contraction. \emph{Edge contraction} means to take an edge $vw$ and replace $v$ and $w$ with a new vertex $x$ adjacent to all vertices originally adjacent to $v$ or $w$. If $H$ is a minor of $G$, say that $G$ has an $H$-minor.

The \emph{Hadwiger number} $\had(G)$ is the order of the largest complete minor of $G$. The Hadwiger number is most relevant to \emph{Hadwiger's Conjecture} \citep{Hadwiger43}, often considered one of the most important unsolved conjectures in graph theory, which states that $\chi(G) \leq \had(G)$. Hadwiger's Conjecture can be seen as an extension of the Four Colour Theorem, since every planar graph has $\had(G) \leq 4$. While the conjecture remains unsolved in general, it has been proved for $\had(G) \leq 5$ \citep{RST-Comb93}.

Given a graph $H$, an \emph{$H$-model} of $G$ is a set of pairwise vertex-disjoint connected subgraphs of $G$, each called a \emph{branch set}, indexed by the vertices of $H$, such that if $vw \in E(H)$, then there exists an edge between the branch sets indexed by $v$ and $w$. If $G$ contains an $H$-model, then repeatedly contract the edges inside each branch set and delete extra vertices and edges to obtain a copy of $H$. Thus if $G$ contains an $H$-model, then $H$ is a minor of $G$. Similarly, if $H$ is a minor of $G$, ``uncontract" each vertex in the minor to obtain an $H$-model of $G$. Models are helpful when dealing with questions relating to minors, since they describe how the $H$-minor ``sits'' in $G$. 

\section{Brambles}
\label{section:ptttbn}
Two subgraphs $A$ and $B$ of a graph $G$ \emph{touch} if $V(A)\cap V(B)\neq\emptyset$, or some edge of $G$ has one endpoint in $A$ and the other endpoint in $B$. A \emph{bramble} in $G$ is a set of connected subgraphs of $G$ that pairwise touch. A set $S$ of vertices in $G$ is a \emph{hitting set} of a bramble $\BB$ if $S$ intersects every element of $\BB$. The \emph{order} of $\BB$ is the minimum size of a hitting set. The \emph{bramble number} of $G$ is the maximum order of a bramble in $G$. Brambles were first defined by \citet{SeymourThomas-JCTB93}, where they were called \emph{screens of thickness $k$}. Seymour and Thomas proved the following result. 

\begin{theorem}[Treewidth Duality Theorem, \citet{SeymourThomas-JCTB93}]
\label{theorem:twdual}
For every graph $G$, $$\tw(G) = \bn(G) - 1.$$
\end{theorem}
\begin{proof}
Here, we present a short proof showing one direction of this result. The other (more difficult) direction can be found in \citep{SeymourThomas-JCTB93}; see \citet{ShortDiestel} for a shorter proof.
Let $\beta$ be a bramble in $G$ of maximum order, and let $T$ be the underlying tree in a tree decomposition of $G$. For a subgraph $A \in \beta$, let $T_{A}$ be the subgraph of $T$ induced by the nodes of $T$ whose bags contain vertices of $A$. Since $A$ is connected, $T_A$ is also connected. Similarly, if $A,B \in \beta$, then since these subgraphs touch, there is a node of $T$ in both $T_A$ and $T_B$. So the set of subtrees $\{ T_A : A \in \beta \}$ pairwise intersect. By the Helly Property of trees, there is some node $x$ that is in all such $T_A$. The bag indexed by $x$ contains a vertex from each $A \in \beta$, so it is a hitting set of $\beta$. Hence that bag has order at least $\bn(G)$, and so $\tw(G) \geq \bn(G) - 1$.
\end{proof}

Note that Theorem~\ref{theorem:twdual} means that the bramble number is equal to the size of the largest bag in a minimum width tree decomposition. 

Brambles are useful for proving a lower bound on the treewidth of a graph. Given a tree decomposition $T$ for a graph $G$, then $\tw(G)$ is at most the width of $T$. Brambles provide the equivalent functionality for the lower bound---given a valid bramble of a graph $G$, it follows that the bramble number is at least the order of that bramble, giving us a lower bound on the treewidth. (For examples of this, see \citet{bodlaenderbramble}, \citet{lucenabramble} and Lemma~\ref{lemma:gridbramble}.)

\section{\texorpdfstring{$k$}{k}-Trees and Chordal Graphs}
\label{section:ptttchord}
In certain applications, such as graph drawing \citep{DMW05,DiGiacomo} or graph colouring \citep{KP-DM08,Albertson-EJC04}, it often suffices to consider only the edge-maximal graphs of a given family to obtain a result. The language of $k$-trees and chordal graphs provides an elegant description of the edge-maximal graphs with treewidth at most $k$.

A vertex $v$ in a graph $G$ is \emph{$k$-simplicial} if it has degree $k$ and its neighbours induce a clique. A graph $G$ is a \emph{$k$-tree} if either:
\begin{itemize*}
\item $G = K_{k+1}$, or
\item $G$ contains a $k$-simplicial vertex $v$ and $G-v$ is also a $k$-tree.
\end{itemize*}
Note that there is some discrepancy over this definition; certain authors use $K_{k}$ in the base case. This means that $K_{k}$ is a $k$-tree, but creates no other changes.
$k$-trees have a strong tie to treewidth; see Lemma~\ref{lemma:ktree}.

A graph is \emph{chordal} if it contains no induced cycle of length at least four. That is, every cycle that is not a triangle contains a chord. \citet{gavril74} showed that the chordal graphs are exactly the intersection graphs of subtrees of a tree $T$. Construct a tree decomposition with underlying tree $T$ as follows. Think of each $v \in V(G)$ as a subtree of $T$; place $v$ in the bags indexed by the nodes of that subtree. It can easily be seen that this is a tree decomposition of $G$ in which every bag is a clique (that is, every possible edge exists), since should two vertices share a bag, then their subtrees intersect and the vertices are adjacent. It also follows that the graph arising from a tree decomposition with all possible edges (that is, two vertices are adjacent if and only if they share a bag) is a chordal graph. Chordal graphs are therefore interesting by being the edge-maximal graphs for a fixed tree-width. The initial definition of $\tw(G)$ by \citet{Halin76} is that $\tw(G)+1$ is equal to the minimum chromatic number of any chordal graph which contains $G$. This is identical to the second equality below, given that chordal graphs are perfect. 

\begin{lemma}[\citep{Bodlaender-TCS98, RS-GraphMinorsII-JAlg86,Scheffler, jvanleeuwen,rose-ktree,arnborg-ktree}]
\label{lemma:ktree}
For every graph $G$,
\begin{align*}
\tw(G) &= \min\{ k: G \text{ is a spanning subgraph of a $k$-tree }\}. \\
 &= \min\{ k: G \text{ is a spanning subgraph of a chordal graph with no $(k+2)$-clique }\}.
\end{align*}
\end{lemma}

\begin{proof}
For simplicity, let $a(G) = \min\{ k: G$ is a spanning subgraph of a $k$-tree$\}$ and $b(G) = \min\{ k: G$ is a spanning subgraph of a chordal graph with no $(k+2)$-clique$\}$.

First, we show $b(G) \leq a(G)$. \citet{FG65} showed that a graph $H$ is chordal if and only if it has a \emph{perfect elimination ordering}; that is, an ordering of the vertex set such that for each $v \in V(H)$, $v$ and the neighbours of $v$ which are after $v$ in the ordering form a clique. If $H$ is an $a(G)$-tree such that $G$ is a spanning subgraph of $H$, then there is a simple perfect elimination ordering for $H$. (Repeatedly delete the $a(G)$-simplicial vertices to obtain $K_{a(G)+1}$, and consider the order of deletion.) So $H$ is chordal. It is clear that each $v$ has only $a(G)$ neighbours after it in this ordering, so $H$ contains no $(a(G)+2)$-clique. (For any clique, consider the first vertex of the clique in the ordering, and note at most $a(G)$ other vertices are in the clique.) Thus $b(G) \leq a(G)$.

Second, we show $a(G) \leq \tw(G)$. Assume for the sake of a contradiction that $G$ is a vertex-minimal counterexample, and say $G$ has treewidth $k$. It is easy to see $a(G) \leq \tw(G)$ when $G$ is complete, so assume otherwise. Let $T$ be a tree decomposition of $G$ with minimum width. By Lemma~\ref{lemma:normal}, assume $T$ is normalised. Note, since $G$ is not complete, $T$ contains more than one bag. Let $G'$ be the graph created by taking $G$ and adding all edges $vw$, where $v$ and $w$ share some bag of $T$. So $G$ is a spanning subgraph of $G'$ and $T$ is a tree decomposition of $G'$ as well as $G$. By the normalisation, there is a vertex $v \in V(G')$ such that $v$ appears in a leaf bag $B$ of $T$ and nowhere else. Hence $v$ has exactly $k$ neighbours in $G'$, which form a clique since they are all in $B$. Since it is smaller than the minimal counterexample, $a(G'-v) \leq \tw(G'-v) \leq k$. Since $G'-v$ contains a $(k+1)$-clique (consider a bag of $T$ other than $B$), it follows $a(G'-v) \geq k$. Thus $a(G'-v)=k$, and $G'-v$ is a spanning subtree of a $k$-tree $H$. Since $v$ is $k$-simplicial in $G'$, it follows $G'$ (and thus $G$) is a spanning subgraph of a $k$-tree, which contradicts our assumption.

Finally, we show that $\tw(G) \leq b(G)$. The graph $G$ is a spanning subgraph of chordal graph $H$ with no $(b(G)+2)$-clique. There is a tree decomposition of $H$ where every bag is a clique; this means it has width at most $b(G)$. This tree decomposition is also a tree decomposition for $G$, so $\tw(G) \leq b(G)$.

Hence, it follows that $b(G) \leq a(G) \leq \tw(G) \leq b(G)$, which is sufficient to prove our desired result.
\end{proof}

\section{Separators}
\label{section:separators}
For a graph $G$, a set $S \subseteq V(G)$, and some $c \in [\frac{1}{2},1)$, a $(k,S,c)$-\emph{separator} is a set $X \subseteq V(G)$ with $|X| \leq k$, such that each component of $G-X$ contains at most $c|S\backslash X|$ vertices of $S$. Note that a $(k,S,c)$-separator is also a $(k,S,c')$-separator for all $c' \geq c$. Define the \emph{separation number} $\sep_c(G)$ to be the minimum integer $k$ such that there is a $(k,S,c)$-separator for all $S \subseteq V(G)$. We also consider the following variant: a $(k,S,c)^*$-\emph{separator} is a set $X \subseteq V(G)$ with $|X| \leq k$ such that each component of $G-X$ contains at most $c|S|$ vertices of $S$. Define $\sep_c^*(G)$ analogously to $\sep_c(G)$, but with respect to these variant separators. It follows from the definition that $\sep_c^*(G) \leq \sep_c(G)$.

Separators can be seen as a generalisation of the ideas presented in the famous planar separator theorem \citep{LT79}, which essentially states that a planar graph $G$ with $n$ vertices contains a $(O(\sqrt{n}),V(G),\frac{2}{3})^*$-separator. Unfortunately, the precise definition of a separator and the separation number is inconsistent across the literature. The above definition is an attempt to unify the existing definitions. \citet{RS-GraphMinorsII-JAlg86} gave the first lower bound on $\tw(G)$ in terms of separators, though they do not use the term. This definition is equivalent to our standard definition but with $c$ fixed at $\frac{1}{2}$. \citet{marxgrohe}, give the above variant definition, with $c$ fixed at $\frac{1}{2}$, and instead call it a \emph{balanced separator}. \citet{Reed97} defines separators using our standard definition, with $c=\frac{2}{3}$. \citet{Bodlaender-TCS98} defines ``type-1" and ``type-2" separators (see below for an explanation), which have variable proportion (i.e. allow for different values of $c$), but are not defined on sets other than $V(G)$. Sometimes \citep{Fox11,marxgrohe,Bodlaender-TCS98} instead of considering components in $G-X$, separators are defined as partitioning the vertex set of $G-X$ into exactly two parts $A$ and $B$, such that no edge has an endpoint in both parts and $|A \cap S|,|B \cap S| \leq c|S|$. (In fact, \citet{Bodlaender-TCS98} uses both this definition and the standard ``components of $G-X$" definition as the difference between type-1 and type-2 separators.) As long as $c \geq \frac{2}{3}$, this is equivalent to considering the components, since Lemma~\ref{lemma:comptrick} and Corollary~\ref{corollary:comptrick} allow partitioning of the components into parts $A$ and $B$.
However, for lower values of $c$ this no longer holds, for example, if $c=\frac{1}{2}$, it is possible that each component contains exactly $\frac{1}{3}$ of the vertices of $S$, so there is no acceptable partition into $A$ and $B$. As a result, $c=\frac{2}{3}$ and $c=\frac{1}{2}$ are the most ``natural" choices for $c$.

Fortunately, $\sep_c(G),\sep_c^*(G),\sep_{c'}(G)$ and $\sep_{c'}^*(G)$ are all tied for all $c,c' \in [\frac{1}{2},1)$.

\citet{RS-GraphMinorsII-JAlg86} proved that $$\sep_{\frac{1}{2}}(G)\leq \tw(G)+1.$$ (Of course, they did not use our notation.) \citet{RS-GraphMinorsII-JAlg86, RS-GraphMinorsXIII-JCTB95} also proved that \begin{equation}\tw(G)+1\leq 4\,\sep_{\frac{2}{3}}(G) - 2.\label{equation:sep1}\end{equation} (\citet{Reed97,Reedktreefinder} gives a more accessible proof of this upper bound.) \citet{grohebook} proved that \begin{equation}
\tw(G) \leq 3\,\sep_{\frac{1}{2}}^{*}(G)-2.\label{equation:sep2}\end{equation} Lemma~\ref{lemma:sepgeqtw} proves a slightly stronger result that replaces the multiplicative constant ``4" by ``3" in equation \eqref{equation:sep1}, and the multiplicative constant ``3" by ``2" in \eqref{equation:sep2}.

First, we prove a useful lemma for dealing with components of a graph.
\begin{lemma}
\label{lemma:comptrick}
For every graph $G$ and for all sets $X, S \subseteq V(G)$ such that each component of $G-X$ contains at most half the vertices of $S \backslash X$, it is possible to partition the components of $G-X$ into at most three parts such that each part contains at most half the vertices of $S \backslash X$. 
\end{lemma}
\begin{proof}
If $G-X$ contains at most three components, the claim follows immediately. Hence assume $G-X$ contains at least four components.
Initially, let each part simply contain a single component. Merge parts as long as doing so does not cause the new part to contain more than half the vertices of $S \backslash X$.
Now if two parts contain more than $\frac{1}{4}$ of the vertices of $S \backslash X$ each, then all other parts (of which there must be at least two) contain, in total, less than $\frac{1}{2}$ of the vertices of $S \backslash X$. Then merge all other parts together, leaving the partition with exactly three parts. Alternatively only one part (at most) contains more than $\frac{1}{4}$ of the vertices of $S \backslash X$. So at least three parts contain at most $\frac{1}{4}$ of the vertices of $S \backslash X$, and so merge two of them. This lowers the number of parts in the partition. As long as there are four or more parts, one of these operations can be performed, so repeat until at most three parts remain.
\end{proof}

\begin{corollary}
\label{corollary:comptrick}
For every graph $G$ and for all sets $X, S \subseteq V(G)$ such that each component of $G-X$ contains at most two-thirds of the vertices of $S \backslash X$, it is possible to partition the components of $G-X$ into at most two parts such that each part contains at most two-thirds of the vertices of $S \backslash X$. 
\end{corollary}

This corollary follows by a very similar argument to Lemma~\ref{lemma:comptrick}.

The following argument is similar to that provided in \citep{RS-GraphMinorsII-JAlg86}.
\begin{lemma}[\citet{RS-GraphMinorsII-JAlg86}]
\label{lemma:sepleqtw}
For every graph $G$ and for all $c \in [\frac{1}{2},1)$, $$\sep_c(G) \leq \tw(G)+1.$$
\end{lemma}
\begin{proof}
Fix $S \subseteq V(G)$ and let $k := \tw(G)+1$. It is sufficient to construct a $(k,S,\frac{1}{2})$-separator for $G$. The graph $G$ has a normalised tree decomposition $T$ with maximum bag size $k$, by Lemma~\ref{lemma:normal}. Consider a pair of adjacent bags $X,Y$. Let $T_X$ and $T_Y$ be the subtrees of $T-XY$ containing bags $X$ and $Y$ respectively. Let $U_X \subseteq V(G)$ be the set of vertices only appearing in bags of $T_X$, and $U_Y$ the set of vertices only appearing in bags of $T_Y$. Then  $U_X, X \cap Y, U_Y$ is a partition of $V(G)$ such that no edge has an endpoint in $U_X$ and $U_Y$. Each component of $G-(X \cap Y)$ is contained entirely within $U_X$ or $U_Y$. Say $Q \subseteq V(G)$ is \emph{large} if $|Q \cap S| > \frac{1}{2}|S \backslash (X \cap Y)|$.

If neither $U_X$ or $U_Y$ is large, then no component of $G-(X \cap Y)$ is large. Hence $X \cap Y$ is a $(|X \cap Y|, S,\frac{1}{2})$-separator. Since $|X \cap Y| \leq |Y| \leq k$, this is sufficient.

Alternatively, for all edges $XY \in E(T)$, exactly one of $U_X$ and $U_Y$ is large. (If both sets are large, then $|S \backslash (X \cap Y)| = |U_X \cap S| + |U_Y \cap S| > |S \backslash (X \cap Y)|$, which is a contradiction.) Orient the edge $XY \in E(T)$ towards $X$ if $U_X$ is large, or towards $Y$ if $U_Y$ is large.

Now there must be a bag $B$ with outdegree 0. If $B$ is a $(|B|, S,\frac{1}{2})$-separator, then since $|B|=k$, the result is achieved. Otherwise, exactly one component $C$ of $G-B$ is large. The vertices of $C$ only appear in the bags of a single subtree of $T-B$. Label that subtree as $T'$, and let $A$ denote the bag of $T'$ adjacent to $B$. Recall there is a partition $V(G)$ into $U_A, A \cap B, U_B$ where $|U_B \cap S| > \frac{1}{2}|S \backslash (A \cap B)|$, since the edge $AB$ is oriented towards $B$. Hence $|U_A \cap S| < \frac{1}{2}|S \backslash (A \cap B)|$. Also note the vertices of $G-B$ that only appear in the bags of $T'$ are exactly the vertices of $U_A$. Hence $C \subseteq U_A$, and $|U_A \cap S| > \frac{1}{2}|S \backslash B|$.

So $\frac{1}{2}|S \backslash B| < |U_A \cap S| < \frac{1}{2}|S \backslash (A \cap B)|$. By our normalisation, $|A \cap B| = |B|-1$. So $|S \backslash B| \geq |S \backslash (A \cap B)|-1$.
Thus $|S \backslash (A \cap B)|-1 < 2|U_A \cap S| < |S \backslash (A \cap B)|$, which is a contradiction since $|S \backslash (A \cap B)|-1$, $2|U_A \cap S|$ and $|S \backslash (A \cap B)|$ are all integers.
\end{proof}

Now we prove the upper bound.
\begin{lemma}
\label{lemma:sepgeqtw}
For every graph $G$, for all $c \in [\frac{1}{2},1)$, $$\bn(G) \leq \frac{1}{1-c}\sep_c^*(G).$$
\end{lemma}
\begin{proof}
Say $\beta$ is an optimal bramble of $G$ with a minimum hitting set $H$. That is, $|H|=\bn(G)$. For the sake of a contradiction, assume that $(1-c)\bn(G) > \sep_c^*(G)$. So there is a $(\sep_c^*(G),H,c)^*$-separator $X$. If $X$ is a hitting set for $\beta$ then $\bn(G) \leq |X| \leq sep_c^*(G) < (1-c)\bn(G)$, which is a contradiction. So $X$ is not a hitting set for $\beta$. Thus some bramble element of $\beta$ is entirely within a component of $G-X$. Only one such component can contain bramble elements. Call this component $C$. Then we can hit every bramble element of $\beta$ with the vertices of $X$ or the vertices of $H$ inside $C$, that is, $X \cup (H \cap V(C))$ is a hitting set. Since $X$ is a $(\sep_c^*(G),H,c)^*$-separator, $|H \cap V(C)| \leq c|H|$. Thus $|X \cup (H \cap V(C))| = |X| + |H \cap V(C)| \leq |X| + c|H| \leq \sep_c^*(G) + c|H| < (1-c)|H| + c|H| = |H|$. Thus $X \cup (H \cap V(C))$ is a hitting set smaller than the minimum hitting set, a contradiction. 
\end{proof}

Hence, from the above it follows that for $c \in [\frac{1}{2},1)$, $$\sep_c^*(G) \leq \sep_c(G) \leq \tw(G) + 1 = \bn(G) \leq \frac{1}{1-c}\sep_c^*(G) \leq \frac{1}{1-c}\sep_c(G).$$

Each of the above inequalities is tight. In particular, the second and third inequalities are tight for $K_n$. For a given $c \in [\frac{1}{2},1)$, if $k,n$ are integers such that $k > \frac{c}{1-c}+1$ and $n \geq \frac{k-1}{1-c}$,  then $\sep_c^*(\psi_{n,k}) = \sep_c(\psi_{n,k}) = n$. (See \citep{thesis} for a proof of this result.) This proves that the first and last inequalities are tight.

Finally, given a graph $G$, let $\sn(G)$ denote the minimum integer $k$ such that, for each subgraph $H$ of $G$, there exists a $(k,|V(H)|,\frac{2}{3})^*$-separator for $H$. The parameter $\sn(G)$ is equivalent to the definition of separation number given by \citet{Fox11}. This version of the separation number is also tied to treewidth. Given that, for every $S \subseteq V(G)$, every $(k,S,c)^*$-separator in $G$ is also a $(k,S,c)^*$-separator in $G[S]$, it follows that $\sn(G) \leq \sep_{\frac{2}{3}}^{*}(G) \leq \tw(G)+1$. The other direction is due to a recently announced result of \citet{balancedsep}.

\begin{lemma}[\citet{balancedsep}]
For every graph $G$, $$\tw(G) \leq 105\,\sn(G).$$
\end{lemma}

\section{Branchwidth and Tangles}
\label{section:tangles}
A \emph{branch decomposition} of a graph $G$ is a pair $(T, \theta)$ where $T$ is a tree with each node having degree 3 or 1, and $\theta$ is a bijective mapping from the edges of $G$ to the leaves of $T$. A vertex $x$ of $G$ is \emph{across} an edge $e$ of $T$ if there are edges $xy$ and $xz$ of $G$ mapped to leaves in different subtrees of $T-e$. The \emph{order} of an edge $e$ of $T$ is the number of vertices of $G$ across $e$. The \emph{width} of a branch decomposition is the maximum order of an edge. Finally, the \emph{branchwidth} $\bw(G)$ of a graph $G$ is the minimum width over all branch decompositions of $G$. Note that if $|E(G)| \leq 1$, there are no branch decompositions of $G$, in which case we define $\bw(G)=0$. \citet{RS-GraphMinorsX-JCTB91} first defined branchwidth, where it was defined more generally for hypergraphs; here we just consider the case of simple graphs.

Tangles were first defined by \citet{RS-GraphMinorsX-JCTB91}. Their definition is in terms of sets of separations of graphs. (Note, importantly, that a \emph{separation} is not the same as a \emph{separator} as defined in Section~\ref{section:separators}.) We omit their definition and instead present the following, initially given by \citet{Reed97}.

A set $\tau$ of connected subgraphs of a graph $G$ is a \emph{tangle} if for all sets of three subgraphs $A,B,C \in \tau$, there exists either a vertex $v$ of $G$ in $V(A \cap B \cap C)$, or an edge $e$ of $G$ such that each of $A,B$ and $C$ contain at least one endpoint of $e$. Clearly a tangle is also a bramble---this is the main advantage of this definition. The \emph{order} of a tangle is equal to its order when viewed as a bramble. The \emph{tangle number} $\tn(G)$ is the maximum order of a tangle in $G$.

When defined with respect to hypergraphs, treewidth and tangle number are tied to the maximum of branchwidth and the size of the largest edge. So for simple graphs, there are a few exceptional cases when $\bw(G) < 2$, which we shall deal with briefly. If $G$ is connected and $\bw(G) \leq 1$, then $G$ contains at most one vertex with degree greater than 1 (that is, $G$ is a star), and $\bn(G)=\tn(G) \leq 2$. Henceforth, assume $\bw(G) \geq 2$. 

\citet{RS-GraphMinorsX-JCTB91} prove the following relation between tangle number and branchwidth; we omit the proof. Instead we show that $\tn(G),\bw(G),\bn(G)$ and $\tw(G)$ are all tied by small constant factors.
\begin{theorem}[\citet{RS-GraphMinorsX-JCTB91}]
\label{theorem:tangbw}
For a graph $G$, if $\bw(G) \geq 2$, then $$\bw(G) = \tn(G).$$
\end{theorem}

\citet{RS-GraphMinorsX-JCTB91} proved that $\bn(G) \leq \frac{3}{2}\tn(G)$. \citet{Reed97} provided a short proof that $\bn(G) \leq 3\,\tn(G)$. Here, we modify Reed's proof to show that $\bn(G) \leq 2\,\tn(G)$.

\begin{lemma}
For every graph $G$,
$$\tn(G)\leq \bn(G)\leq 2\,\tn(G).$$
\end{lemma}

\begin{proof}

Since every tangle is also a bramble, $\tn(G) \leq \bn(G)$.

To prove that $\bn(G) \leq 2\,\tn(G)$, let $k:=\bn(G)$, and say $\beta$ is a bramble of $G$ of order $k$. Consider a set $S \subseteq V(G)$ with $|S| < k$. If two components of $G-S$ entirely contain a bramble element of $\beta$, then those two bramble elements do not touch. Alternatively, if no component of $G-S$ entirely contains a bramble element, then all bramble elements use a vertex in $S$, and $S$ is a hitting set of smaller order than the minimum hitting set. Thus exactly one component $S'$ of $G-S$ entirely contains a bramble element of $\beta$. Clearly, $V(S') \cap S = \emptyset$.

Define $\tau := \{ S' : S \subseteq V(G), |S| < \frac{k}{2} \}$. To prove that $\tau$ is a tangle, let $T_1, T_2, T_3$ be three elements of $\tau$. Say $T_i = S'_i$ for each $i$. Since $|S_1 \cup S_2| < k$, some bramble element $B_1$ of $\beta$ does not intersect $S_1 \cup S_2$. Similarly, some bramble element $B_2$ does not intersect $S_2 \cup S_3$. Since $B_1$ does not intersect $S_1$, it is entirely within one component of $G-S_1$, that is, $B_1 \subseteq T_1$. Similarly, $B_1 \subseteq T_2$ and $B_2 \subseteq T_2 \cap T_3$. Since $B_1,B_2 \in \beta$, they either share a vertex $v$, or there is an edge $e$ with one endpoint in $B_1$ and the other in $B_2$. In the first case, $v \in V(T_1 \cap T_2 \cap T_3)$. In the second case, one endpoint of $e$ is in $T_1 \cap T_2$, the other in $T_2 \cap T_3$. It follows that $\tau$ is a tangle. The order of $\tau$ is at least $\frac{k}{2}$, since a set $X$ of size less than $\frac{k}{2}$ has a defined $X' \in \tau$, and so $X$ does not intersect all subgraphs of $\tau$. Then $\tn(G) \geq \frac{k}{2}$. 
\end{proof}

We now provide a proof for a direct relationship between branchwidth and treewidth. Note again these proofs are modified versions of those in \citep{RS-GraphMinorsX-JCTB91}. 
\begin{lemma}[\citet{RS-GraphMinorsX-JCTB91}]
\label{lemma:twbranch}
For a graph $G$, if $\bw(G) \geq 2$ then $$\bw(G) \leq \tw(G)+1 \leq \frac{3}{2}\,\bw(G).$$
\end{lemma}
\begin{proof}
We prove the second inequality first. Assume no vertex is isolated. Let $k:= \bw(G)$, and let $(T,\theta)$ be a branch decomposition of order $k$. We construct a tree decomposition with $T$ as the underlying tree, and where $B_x$ will denote the bag indexed by each node $x$ of $T$. A node $x$ in $T$ has degree 3 or 1.  If $x$ has degree 1, then let $B_{x}$ contain the two endpoints of $e = \theta^{-1}(x)$. If $x$ has degree 3, then let $B_{x}$ be the set of vertices that are across at least one edge incident to $x$. We now show that this is a tree decomposition. Every vertex appears at least once in the tree decomposition. Also, for every edge $vw \in E(G)$, the bag of the leaf node $\theta(vw)$ contains both $v$ and $w$. If we consider vertex $v \in V(G)$ incident with $vw$ and $vu$, then $v$ is across every edge in $T$ on the path from $\theta(vw)$ to $\theta(vu)$. Thus, $v$ is in every bag indexed by a node on that path. Such a path exists for all neighbours $w,u$ of $v$. It follows that the subtree of nodes indexing bags containing $v$ form a subtree of $T$. Thus $(T, (B_x)_{x \in V(T)})$ is a tree decomposition of $G$. A bag indexed by a leaf node has size 2. If $x$ is not a leaf, then $B_x$ contains the vertices that are across at least one edge incident to $x$. Suppose $v$ is across exactly one such edge $e$. Then there exists $\theta(vw)$ and $\theta(vu)$ in different subtrees of $T-e$. Without loss of generality, $\theta(vw)$ is in the subtree containing $x$. But then the path from $x$ to $\theta(vw)$ uses one of the other two edges incident to $x$. Hence if $v$ is in $B_{x}$ then $v$ is across at least two edges incident to $x$. If the sets of vertices across the three edges incident to $x$ are $A,B$ and $C$ respectively, then $|A| + |B| + |C| \geq 2|B_{x}|$. But $|A|+|B|+|C| \leq 3k$. Therefore, regardless of whether $x$ is a leaf, $|B_{x}| \leq \max\{2,\frac{3}{2}k\} =\frac{3}{2}k$ (since $k \geq 2$). Therefore $\tw(G)+1 \leq \frac{3}{2}k$.

Now we prove the first inequality. Let $k:=\tw(G)+1$. Hence there exists a tree decomposition $(T, (B_x)_{x \in V(T)})$ with maximum bag size $k$; choose this tree decomposition such that $T$ is node-minimal, and such that the subtree induced by $\{ x \in V(T) : v \in B_x\}$ is also node-minimal for each $v \in V(G)$. If $k < 2$, then $G$ contains no edge, and $\bw(G) = 0$. Now assume $k \geq 2$ and $E(G) \neq \emptyset$. Since the first inequality is trivial when $G$ is complete, we assume otherwise, and thus $T$ is not a single node. 

Note the following facts: if $x$ is a node of $T$ with degree 2, then there exists some pair of adjacent vertices $v,w$ such that $B_x$ is the only bag containing $v$ and $w$. (Otherwise, $T$ would violate the minimality properties.) Similarly, if $x$ is a leaf node, then there exists some $v \in V(G)$ such that $B_x$ is the only bag containing $v$. The bag $B_x$ also contains the neighbours of $v$, but nothing else. 

Now, for every edge $vw \in E(G)$, choose some bag $B_x$ containing $v$ and $w$. Unless $x$ is a leaf with $B_x = \{v,w\}$, add to $T$ a new node $y$ adjacent to $x$, such that $B_y = \{v,w\}$. Clearly $(T, (B_x)_{x \in V(T)})$ is still a tree decomposition of the same width. From our above facts, every leaf node is either newly constructed or was already of the form $B_x = \{v,w\}$. Also, every node that previously had degree 2 now has higher degree. A node that was previously a leaf either remains a leaf, or now has degree at least 3. So no node of the new $T$ has degree 2.

If a node $x$ has degree greater than 3, then delete the edges from $x$ to two of its neighbours (denoted $y,z$), and add to $T$ a new node $s$ adjacent to $x,y$ and $z$. Let $B_s := B_x \cap (B_y \cup B_z)$. Clearly this is still a tree decomposition of the same width. Now the degree of $x$ has been reduced by 1, and the new node has degree 3. Repeat this process until all nodes have either degree 3 or 1.

Since each leaf bag contains exactly the endpoints of an edge (and no edge has both endpoints in more than one leaf), there is a bijective mapping $\theta$ that takes $vw \in E(G)$ to the leaf node containing $v$ and $w$. Together with $T$, this gives a branch decomposition of $G$. If $xy \in E(T)$, then all edges of $G$ across $xy$ are in $B_x \cap B_y$. So the order of this branch decomposition is at most $k$. Thus $\bw(G) \leq \tw(G)+1$.

(Note that our minimality properties would imply that $|B_x \cap B_y| < k$, however converting the tree to ensure that all nodes have degree 3 or 1 does not necessarily maintain this.) 
\end{proof}

\citet{RS-GraphMinorsX-JCTB91} showed the bounds in Lemma~\ref{lemma:twbranch} are tight. The upper bound on $\tw(G)$ in Lemma~\ref{lemma:twbranch} is tight for $K_n$ when $n$ is divisible by 3, since $\tw(K_n)=n-1$ and $\bw(K_n) = \tn(K_n) = \frac{2}{3}n$.
The lower bound on $\tw(G)$ is tight when $n \geq 4$ and $G$ is the graph $K_{n,n}$ minus a perfect matching. In this case $\tw(G) +1 = \bw(G) = \tn(G) = n$.

\section{Tree Products}
\label{section:treeproducts}
For a tree $T$, let $T \cdot K_k$ denote the lexicographic product of $T$ with $K_k$. That is, $T \cdot K_k$ is the graph created by taking $T$ and replacing each vertex with a clique of $k$ vertices, and replacing each edge with all possible edges between the two new cliques. The \emph{lexicographic tree product number} of $G$, denoted $\ltp(G)$, is the minimum integer $k$ such that $G$ is a minor of the graph $T \cdot K_k$ for some tree $T$.

\begin{lemma}
\label{lemma:ltp}
For every graph $G$, $$\ltp(G)-1\leq \tw(G)\leq2\,\ltp(G)-1.$$
\end{lemma}

\begin{proof}
First we prove that $\ltp(G)\leq\tw(G)+1$. Consider a tree decomposition of $G$ with width $k:=\tw(G)$ whose underlying tree is $T$. Clearly, $G$ is a minor of $T \cdot K_{k+1}$. Thus $\ltp(G)\leq k+1$.

Now we prove that $\tw(G)\leq2\ltp(G)-1$. Let $T$ be a tree such that $G$ is a minor of $T \cdot K_k$ where $k:=\ltp(G)$. For each vertex $v$ of $T$ let $K_v$ be the copy of $K_k$ that replaces $v$ in the construction of $T \cdot K_k$. Let $T'$ be the tree obtained from $T$ by subdividing each edge. Now we construct a tree decomposition of $T \cdot K_k$ whose underlying tree is $T'$. For each vertex $v$ of $T$, let the bag at $v$ consist of $K_v$. For each edge $vw$ of $T$ subdivided by vertex $x$, let the bag at $x$ consist of $K_v\cup K_w$. Thus each edge of $T \cdot K_k$ is in some bag, and the set of bags that contain each vertex of $T \cdot K_k$ form a connected subtree of $T'$. Hence we have a tree decomposition of $T'$. Each bag has size at most $2k$. Hence $\tw(T \cot K_k)\leq 2k-1$. (In fact, $\tw(T \cdot K_k)= 2k-1$ since $T \cdot K_k$ contains $K_{2k}$.)\ Thus every minor of $T'$, including $G$, has treewidth at most $2k-1$. 
\end{proof}

If $T$ is a tree, let $T^{(k)}$ denote the Cartesian product of $T$ with $K_k$. That is, the graph with vertex set $\{(x,i): x \in T, i \in \{1, \dots, k\}\}$ and with an edge between $(x,i)$ and $(y,j)$ when $x=y$, or when $xy \in E(T)$ and $i=j$. Then define the \emph{Cartesian tree product number} of $G$, $\ctp(G)$, to be the minimum integer $k$ such that $G$ is a minor of $T^{(k)}$. 
The parameter $\ctp(G)$ was first defined by \citet{holst} and \citet{yves}, however they did not use that name or notation, instead calling it \emph{largeur d'arborescence}, and denoting it $\textsf{\textup{la}}(G)$. They also proved the following result. We provide a different proof.

\begin{lemma}[\citet{holst,yves}]
\label{lemma:ctp}
For every graph $G$, $$\ctp(G)-1 \leq \tw(G) \leq \ctp(G).$$
\end{lemma}
\begin{proof}
Let $k:= \tw(G)$. By Lemma~\ref{lemma:ktree}, $G$ is the spanning subgraph of a chordal graph $G'$ that contains a $(k+1)$-clique but no $(k+2)$-clique. Let $(T, (B_x \subseteq V(G))_{x \in V(T)})$ be a minimum width tree decomposition of $G'$. This has width $k$ and is also a tree decomposition of $G$. To prove the first inequality, it is sufficient to show that $G$ is a minor of $T^{(k+1)}$. Let $c$ be a $(k+1)$-colouring of $G'$. (It is well known that chordal graphs are perfect.) For each $v \in V(G)$, define the connected subgraph $R_v$ of $T^{(k+1)}$ such that $R_v := \{(x,c(v)): v \in B_x\}$. If $(x,i) \in V(R_v) \cap V(R_w)$ then both $v$ and $w$ are in $B_x$ and $c(v)=c(w)=i$. But if $v$ and $w$ share a bag then $vw \in E(G')$, which contradicts the vertex colouring $c$. So the subgraphs $R_v$ are pairwise disjoint, for all $v \in V(G)$. If $vw \in E(G)$, then $v$ and $w$ share a bag $B_x$. Hence there is an edge $(x,c(v))(x,c(w))$ between the subgraphs $R_v$ and $R_w$. Hence the $R_v$ subgraphs form a $G$-model of $T^{(k+1)}$.

Now we prove the second inequality. Let $k:= \ctp(G)$, and choose tree $T$ such that $G$ is a minor of $T^{(k)}$. Since $\tw(G) \leq \tw(T^{(k)})$, it is sufficient to show that $\tw(T^{(k)}) \leq k$. Let $T'$ be the tree $T$ with each edge subdivided $k$ times. Label the vertices created by subdividing $xy \in E(T)$ as $xy(1), \dots, xy(k)$, such that $xy(1)$ is adjacent to $x$ and $xy(k)$ is adjacent to $y$. Construct $(T', (B_x \subseteq V(G))_{x \in V(T')})$ as follows. For a vertex $x \in T$, let $B_x = \{(x,i) | i \in \{1, \dots, k\}\}$. For a subdivision vertex $xy(j)$, let $B_{xy(j)} = \{(x,i),(y,i') | 1 \leq i' \leq j \leq i \leq k\}$. This is a valid tree decomposition with maximum bag size $k+1$. Hence $\tw(T^{(k)}) \leq k$ as required.
\end{proof}

The first inequalities in Lemmas~\ref{lemma:ltp} and \ref{lemma:ctp} are tight. Let $k,n$ be integers such that $n \geq 3$. Then the first inequalities in Lemma~\ref{lemma:ltp} and Lemma~\ref{lemma:ctp} are tight for $\psi_{n,k}$ \citep{thesis}. (Also see \citet{MS} for a similar result.) The second inequalities in Lemmas~\ref{lemma:ltp} and \ref{lemma:ctp} are tight for $K_n$ (for Lemma~\ref{lemma:ltp}, ensure that $n$ is even).

\section{Linkedness}
\label{section:linkedness}
\citet{Reed97} introduced the following definition. For a positive integer $k$, a set $S$ of vertices in a graph $G$ is \emph{$k$-linked} if for every set $X \subseteq V(G)$ such that $|X| < k$ there is a component of $G-X$ that contains more than half of the vertices in $S$. The \emph{linkedness} of $G$, denoted by $\link(G)$, is the maximum integer $k$ for which $G$ contains a $k$-linked set. Linkedness is used by \citet{Reed97} in his proof of the Grid Minor Theorem.

\begin{lemma}[\citet{Reed97}]
\label{lemma:lemmalink}
For every graph $G$, $$\link(G)\leq \bn(G)\leq2\,\link(G).$$
\end{lemma}

\begin{proof}
First we prove that $\link(G)\leq \bn(G)$. Let $k:=\link(G)$. Let $S$ be a $k$-linked set of vertices in $G$. Thus, for every set $X$ of fewer than $k$ vertices there is a component of $G-X$ that contains more than half of the vertices in $S$. This component is unique. Call it the \emph{big} component. Let $\beta$ be the set of big components (taken over all such sets $X$). Clearly, any two elements of $\beta$ intersect at a vertex in $S$. Hence $\beta$ is a bramble. Let $H$ be a hitting set for $\beta$. If $|H|<k$ then (by the definition of $k$-linked) there is a component of $G-H$ that contains more than half of the vertices in $S$, implying $H$ does not hit some big component. This contradiction proves that $|H|\geq k$. Hence $\beta$ is a bramble of order at least $k$. Therefore $\bn(G)\geq k=\link(G)$. 

Now we prove that $\bn(G)\leq 2\,\link(G)$. Assume for the sake of a contradiction that $\bn(G) > 2\,\link(G)$. Let $k := \link(G)$, so $G$ is not $(k+1)$-linked. Let $H$ be a minimum hitting set for a bramble $\beta$ of $G$ of largest order. Since $H$ is not $(k+1)$-linked, there exists a set $X$ of order at most $k$ such that no component of $G-X$ contains more than half of the vertices in $H$. Note that at most one component of $G-X$ can entirely contain a bramble element of $\beta$ (otherwise two bramble elements do not touch). If no component of $G-X$ entirely contains a bramble element of $\beta$, then $X$ is a hitting set for $\beta$ of order $|X| \leq k < \frac{1}{2}\bn(G)$, which contradicts the order of the minimum hitting set. Finally, if a component of $G-X$ entirely contains some bramble element of $\beta$, then let $H' \subset H$ be the set of vertices of $H$ in that component. Now $H'$ intersects all of the bramble elements contained in the component (since those bramble elements do not intersect any other vertices of $H$), and $X$ intersects all remaining bramble elements, as in the previous case. Thus, $H' \cup X$ is a hitting set for $\beta$. However, $|X| \leq k < \frac{1}{2}\bn(G)$, and by the choice of $X$, $|H'| \leq \frac{1}{2}|H| = \frac{1}{2}\bn(G)$. So $|H' \cup X| = |H'| + |X| < \bn(G)$, again contradicting the order of the minimum hitting set.
\end{proof}

When $n$ is even $\link(K_n) = \frac{n}{2}$, so the second inequality is tight. The first inequality is tight since $\link(\psi_{n,k}) = \bn(\psi_{n,k}) = n$ when $k \geq 2$ and $n \geq 3$ \citep{thesis}.

\section{Well-linked and \texorpdfstring{$k$}{k}-Connected Sets}
For a graph $G$, a set $S \subseteq V(G)$ is \emph{well-linked} if for every pair $A,B \subseteq S$ such that $|A|=|B|$, there exists a set of $|A|$ vertex-disjoint paths from $A$ to $B$. If we can ensure these vertex-disjoint paths also have no internal vertices in $S$, then $S$ is \emph{externally-well-linked}. The notion of a well-linked set was first defined by \citet{Reed97}, while a similar definition was used by \citet{RST-JCTB94}. Reed also described externally-well-linked sets in the same paper (but did not define it explicitly) and stated but did not prove that $S$ is well-linked if and only if $S$ is externally-well-linked. We provide a proof below. 
The \emph{well-linked number} of $G$, denoted $\wl(G)$, is the size of the largest well-linked set in $G$.

\begin{lemma}[\citet{Reed97}]
\label{lemma:extwl}
$S$ is well-linked if and only if $S$ is externally-well-linked.
\end{lemma}
\begin{proof}
It should be clear that if $S$ is externally-well-linked that $S$ is well-linked, so we prove the forward direction. Let $S \subseteq V(G)$ be well-linked. It is sufficient to show that for all $A,B \subseteq S$ with $|A|=|B|$ there are $|A|$ vertex-disjoint paths from $A$ to $B$ that are internally disjoint from $S$. Define $C := S \backslash (A \cup B)$ and $A' := A \cup C$ and $B' := B \cup C$. Now $S = A' \cup B'$. Since $S$ is well-linked and $|A'| = |B'|$, there are $|A'|$ vertex-disjoint paths between $A'$ and $B'$. Each such path uses exactly one vertex from $A'$ and one vertex from $B'$. Thus, if $v \in C \subseteq A \cap B$, then the path containing $v$ must simply be the singleton path $\{v\}$. Thus this path set contains a set of singleton paths for each vertex of $C$ and, more importantly, a set of paths starting in $A' \backslash C = A$ and ending at $B' \backslash C = B$. Since every vertex of $S$ is in either $A'$ or $B'$, and each path starts at a vertex in $A'$ and ends at one in $B'$, no internal vertex of these paths is in $S$. This is the required set of disjoint paths from $A$ to $B$ that are internally disjoint from $S$.
\end{proof}

\citet{Reed97} proved that $\bn(G) \leq \wl(G) \leq 4\,\bn(G)$. We provide Reed's proof of the first inequality and modify the proof of the second to give:
\begin{lemma}
\label{lemma:welllink}
For every graph $G$, $$\bn(G) \leq \wl(G) \leq 3\,\link(G) \leq 3\,\bn(G).$$ 
\end{lemma}

\begin{proof}
We first prove that $\bn(G) \leq \wl(G)$. Assume for the sake of a contradiction that $\wl(G) < \bn(G)$. Let $\beta$ be a bramble of largest order, and $H$ a minimal hitting set of $\beta$. Thus $H$ is not well-linked (since $|H| = \bn(G) > \wl(G)$). Choose $A,B \subseteq H$ such that $|A|=|B|$ but there are not $|A|$ vertex-disjoint paths from $A$ to $B$. By Menger's Theorem, there exists a set of vertices $C$ with $|C| < |A|$ such that after deleting $C$, there is no $A$-$B$ path in $G$. Now consider a bramble element of $\beta$.
If two components of $G-C$ entirely contain bramble elements, then those bramble elements cannot touch. Thus, it follows that at most one component of $G-C$ entirely contains some bramble element. Label this component $C'$; if no such component exists label $C'$ arbitrarily. Since $C'$ does not contain vertices from both $A$ and $B$, without loss of generality we assume $A \cap C' = \emptyset$. Thus all bramble elements entirely within $C'$ are hit by vertices of $H \backslash A$, and all others are hit by $C$. So $(H \backslash A) \cup C$ is a hitting set for $\beta$, but $|(H \backslash A) \cup C| = |H|-|A|+|C| < |H|$, contradicting the minimality of $H$. Hence $\bn(G) \leq \wl(G)$.

Now we show that $\wl(G) \leq 3\,\link(G)$. For the sake of a contradiction, say $3\,\link(G)< \wl(G)$. Define $k := \frac{1}{3}\wl(G)$. Let $S$ be the largest well-linked set. That is, $|S| = \wl(G)$. By Lemma~\ref{lemma:extwl} $S$ is externally-well-linked. The set $S$ is not $\ceil{k}$-linked since $\link(G) < \ceil{k}$. Hence there exists a set $X \subseteq V(G)$ with $|X| < \ceil{k}$ such that $G-X$ contains no component containing more than $\frac{1}{2}|S|$ vertices of $S$. Since $|X|$ is an integer, $|X| < k$. Let $a := |X \cap S|$.

Using an argument similar to Lemma~\ref{lemma:comptrick}, the components of $G-X$ can be partitioned into two or three parts, each with at most $\frac{1}{2}|S|$ vertices of $S$. Some part contains at least a third of the vertices of $S \backslash X$. Let $A$ be the set of vertices in $S$ contained in that part, and let $B$ be the set of vertices in $S$ in the other parts of $G-X$. Now $\frac{1}{2}|S| \geq |A| \geq \frac{1}{3}|S \backslash X| = \frac{1}{3}(|S| - a)$, and so $|B| \geq |S| - |S \cap X| - |A| \geq |S| - a - \frac{1}{2}|S|$. Remove vertices arbitrarily from the largest of $A$ and $B$ until these sets have the same order. Hence $|A|=|B|$ and $|A| \geq \min\{\frac{1}{3}(|S| - a), \frac{1}{2}|S| - a\}$. Since $A,B \subseteq S$ and $S$ is externally-well-linked, there are $|A|$ vertex-disjoint paths from $A$ to $B$ with no internal vertices in $S$. Since $A$ and $B$ are in different components of $G-X$, these paths must use vertices of $X$, but more specifically, vertices of $X \backslash S$. Thus there are at most $|X \backslash S|$ such paths. Thus $|A| \leq |X \backslash S| < k-a$.

Either $\frac{1}{3}(|S|-a) \leq |A| < k-a$ or $\frac{1}{2}|S| - a \leq |A| < k-a$, so $|S| < 3k$. However, $|S|=\wl(G)=3k$, which is a contradiction.

The final inequality follows from Lemma~\ref{lemma:lemmalink}. 
\end{proof}

The first inequality in Lemma~\ref{lemma:welllink} is tight since $\bn(K_n) = \wl(K_n) = n$. We do not know if the second inequality is tight, but $\wl(G) \leq 2\,\bn(G)-2$ would be best possible since $\bn(K_{2n,n}) = n+1$ and $\wl(K_{2n,n}) = 2n$ (the larger part is the largest well-linked set).  

\citet{diestelHC} defined the following: $S \subseteq V(G)$ is \emph{$k$-connected} in $G$ if $|S| \geq k$ and for all subsets $A,B \subseteq S$ with $|A| = |B| \leq k$, there are $|A|$ vertex-disjoint paths from $A$ to $B$. If we can ensure these vertex-disjoint paths have no internal vertex or edge in $G[S]$, then $S$ is \emph{externally $k$-connected}. This notion was used in \citep{diestelHC} to prove a short version of the grid minor theorem. 

Note the obvious connection to well-linked sets: $X$ is well-linked if and only if $X$ is $|X|$-connected. Also note that \citet{Diestel00a}, in his treatment of the grid minor theorem, provides a slightly different formulation of externally $k$-connected sets, which only requires vertex-disjoint paths between $A$ and $B$ when they are disjoint subsets of $S$. These definitions are equivalent, which can be proven using a similar argument as in Lemma~\ref{lemma:extwl}. \citet{Diestel00a} also does not use the concept of $k$-connected sets, just externally $k$-connected sets.

\citet{diestelHC} prove the following, but due to its similarity between $k$-connected sets and well-linked sets, we omit the proof.
\begin{lemma}[\citet{diestelHC}]
If $G$ has $\tw(G) < k$ then $G$ contains $(k+1)$-connected set of size $\geq 3k$.
If $G$ contains no externally $(k+1)$-connected set of size $\geq 3k$, then $\tw(G) < 4k$.
\end{lemma}

\section{Grid Minors}
\label{section:gridminors}
A key part of the Graph Minor Structure Theorem is as follows: given a fixed planar graph $H$, there exists some integer $r_{H}$ such that every graph with no $H$-minor has treewidth at most $r_{H}$. This cannot be generalised to when $H$ is non-planar, since there exist planar graphs, the grids, with unbounded treewidth. (By virtue of being planar, the grids do not contain a non-planar $H$ as a minor.) In fact, since every planar graph is the minor of some grid, it is sufficient to just consider the grids, which leads to the Grid Minor Theorem:

\begin{theorem}[\citet{RS-GraphMinorsV-JCTB86}]
For each integer $k$ there is a minimum integer $f(k)$ such that every graph with treewidth at least $f(k)$ contains the $k \times k$ grid as a minor.
\end{theorem}

All of our previous sections have provided parameters with linear ties to treewidth. However, the order of the largest grid minor is not linearly tied to treewidth.
The initial bound on $f(k)$ by \citet{RS-GraphMinorsV-JCTB86} was an iterated exponential tower. Later, \citet{RST-JCTB94} improved this to $f(k) \leq 20^{2k^5}$. They also note, by use of a probabilistic argument, that $f(k) \geq \Omega(k^2\log k)$. \citet{diestelHC} obtained an upper bound of $2^{5k^{5}\log k}$, which is actually slightly worse than the bound provided by Robertson, Seymour and Thomas, but with a more succinct proof. \citet{KenandYusuke} proved that $f(k) \leq 2^{O(k^2 \log k)}$, and \citet{SeymourLeaf} proved that $f(k) \leq 2^{O(k\log k)}$. The function $f(k)$ was first shown to be polynomial by \citet{CC14}, who showed $f(k) \leq O(k^{98} \text{polylog}~k)$. A recent result of \citet{Chuzhoy15} improves this to $f(k) \leq O(k^{36} \text{polylog}~k)$. Together with the following lower bound, this implies that treewidth and the order of the largest grid-minor are polynomially tied.

\begin{lemma}[Folklore]
\label{lemma:gridbramble}
If $G$ contains a $k \times k$ grid minor, then $\tw(G) \geq k$.
\end{lemma}
\begin{proof}
If $H$ is a minor of $G$ then $\tw(H) \leq \tw(G)$. Thus it suffices to prove that the $k \times k$ grid $H$ has treewidth at least $k$, which is implied if $\bn(H) \geq k+1$. Consider $H$ drawn in the plane. For a subgraph $S$ of $H$, define a \emph{top vertex} of $S$ in the obvious way. (Note it is not necessarily unique.) Similarly define \emph{bottom vertex}, \emph{left vertex} and \emph{right vertex}. Let subgraph $H'$ of $H$ be the top-left $(k-1) \times (k-1)$ grid in $H$. A \emph{cross} is a subgraph containing exactly one row and column from $H'$, and no vertices outside $H'$. Let $X$ denote the bottom row of $H$, and $Y$ the right column without its bottom vertex. Let $\beta:= \{X,Y, \text{ all crosses}\}$. A pair of crosses intersect in two places. There is an edge from a bottom vertex of a cross to $X$ and a right vertex of a cross to $Y$. There is also an edge from the right vertex of $X$ to the bottom vertex of $Y$. Hence $\beta$ is a bramble. If $Z$ is a hitting set for $\beta$, it must contain $k-1$ vertices of $V(H')$, for otherwise a row and column are not hit, and so a cross is not hit. The set $Z$ must also contain two other vertices to hit $X$ and $Y$. So $|Z| \geq k+1$, as required. 
\end{proof}

\section{Grid-like Minors}
\label{section:gridlikeminors}
A \emph{grid-like-minor of order $t$} of a graph $G$ is a set of paths $\PP$ in $G$ with a bipartite intersection graph that contains a $K_t$-minor. Note that if the intersection graph of $\PP$ is partitioned $A$ and $B$, then we can think of the set of paths $A$ as being the ``rows" of the ``grid", and the set $B$ being the ``columns". Also note that an actual $k \times k$ grid gives rise to a set $\PP$ with an intersection graph $K_{k,k}$ and as such contains a complete minor of order $k$+1. Let $\glm(G)$ be the maximum order of a grid-like-minor of $G$. Grid-like-minors were first defined by \citet{ReedWood-EuJC} as a weakening of a grid minor; see Section~\ref{section:gridminors}. As a result of this weakening, it is easier to tie $\glm(G)$ to $\tw(G)$. This notion has also been applied to prove computational intractability results in monadic second order logic; see \citet{MR2904950,Ganian14} and \citet{KT10,MR2809681}.

The following definitions were independently introduced by \citet{Fox11}\footnotemark[2] and \citet{Pedersen11}. \footnotetext[2]{\citet{Fox11} states that the definitions were independently introduced by Seymour.}Given a graph $G$, consider a bramble $\beta$ together with a function $w$ which assigns a weight to each subgraph in $\beta$, such that for any vertex $v$, the sum of the weights of the bramble elements containing $v$ is at most $1$. Let $h(\beta,w)=\sum_{X \in \beta} w(X)$. The \emph{fractional Hadwiger number} of $G$, denoted $\had_f(G)$, is the maximum of $h(\beta,w)$ over all $\beta,w$ where the weights assigned by $w$ are non-negative real numbers. For a positive integer $r$, the \emph{$r$-integral Hadwiger number} of $G$, denoted $\had_r(G)$, is the maximum of $h(\beta,w)$ over all $\beta,w$ where the weights assigned by $w$ are integer multiples of $\frac{1}{r}$. It is clear that $\had_f(G) \geq \had_r(G)$ for every $G$ and positive integer $r$.
As an example, the branch sets of a $K_{\had(G)}$-minor form a bramble, and we set the weight of each branch set to be $1$. Thus $\had_f(G) \geq \had_r(G) \geq \had(G)$ for all positive integers $r$. 

The graph $G\square K_2$ (that is, the Cartesian product of $G$ with $K_2$) consists of two disjoint copies of $G$ with an edge between corresponding vertices in the two copies. 

\citet{Fox11} proved that $\had(G\square K_r) = r\,\had_r(G) \leq r\,\had_f(G)$. \citet{ReedWood-EuJC} proved that $\glm(G) \leq \had(G\square K_2)$. Here we provide a proof.
\begin{lemma}
\label{lemma:glmhad}
  For every graph $G$ and integer $r\geq2$,
$$\glm(G) \leq \had(G\square K_2) \leq 3\,\had_r(G) \enspace,$$
and if $r$ is even then
$$\glm(G) \leq \had(G\square K_2) \leq 2\,\had_r(G) \leq 2\,\had_f(G)\enspace.$$
\end{lemma}

\begin{proof}
Let $t:=\glm(G)$. It suffices to show there exists a $K_t$-model in $G\square K_2$. Label the vertices of $K_2$ as $1$ and $2$, so a vertex of $G \square K_2$ has the form $(v,i)$ where $v \in V(G)$ and $i \in \{1,2\}$. If $S$ is a subgraph of $G$, define $(S,i)$ to be the subgraph of $G \square K_2$ induced by $\{(v,i) | v \in S\}$. Let $H$ be the intersection graph of a set of paths $\PP$ with bipartition $A,B$, such that $H$ has a $K_t$-minor. For each $P \in \PP$, let $P' := (P,i)$ where $i=1$ if $P \in A$, and $i=2$ if $P \in B$.

If $PQ \in E(H)$, then without loss of generality $P \in A$ and $Q \in B$, and there exists a vertex $v$ such that $v \in V(P) \cap V(Q)$. Then the edge $(v,1)(v,2) \in E(G \square K_2)$ has one endpoint in $P'$ and the other in $Q'$. So $P' \cup Q'$ is connected.

Let $X_1, \dots, X_t$ be the branch sets of a $K_t$-model in $H$. Define $X_i':= \bigcup_{P \in X_i} P'$. Now each $X_i'$ is connected. It is sufficient to show, for $i \neq j$, that $V(X_i' \cap X_j') = \emptyset$ and there exists an edge of $G \square K_2$ with one endpoint in $X_i'$ and the other in $X_j'$.
If there exists $v \in V(X_i' \cap X_j')$ then there exists $P'$ such that $v \in P'$ and $P' \in X_i' \cap X_j'$. But then $P \in X_i \cap X_j$, which is a contradiction. So $V(X_i' \cap X_j') = \emptyset$.
Also, since $X_1, \dots, X_t$ is a $K_t$-model of $H$, there exists some $PQ \in E(H)$ such that $P \in X_i$ and $Q \in X_j$. From above, there exists an edge between $P'$ and $Q'$ in $G \square K_2$, which is sufficient.

For the other inequalities, let $X_1,\dots,X_t$ be the branch sets of a $K_t$-minor in $G\square
  K_2$, where $t:=\had(G\square K_2)$. Let $X'_i$ be the projection of
  $X_i$ into the first copy of $G$. Thus $X'_i$ is a connected subgraph
  of $G$. If $X_i$ and $X_j$ are joined by an edge between the two
  copies of $G$, then $X'_i$ and $X'_j$ intersect. Otherwise, $X_i$
  and $X_j$ are joined by an edge within one of the copies $G$, in
  which case, $X'_i$ and $X'_j$ are joined by an edge in $G$.  Thus
  $X'_1,\dots,X'_t$ is a bramble in $G$. Weight each $X'_i$ by
  $\floor{\frac{r}{2}}/r$, which is at least $\frac{1}{3}$ and at most
  $\frac{1}{2}$.  Since $X_1,\dots,X_t$ are pairwise disjoint, each
  vertex of $G$ is in at most two of $X'_1,\dots,X'_t$. Hence the sum
  of the weights of $X'_i$ that contain a vertex $v$ is at most
  $1$. Hence $\had_r(G)$ is at least the total weight, which is at
  least $\frac{t}{3}$. That is, $\had(G\square K_2)\leq 3\had_r(G)$.
  If $r$ is even then the total weight equals $\frac{t}{2}$ and
  $\had(G\square K_2)\leq 2\,\had_r(G)$, which is at most
  $2\,\had_f(G)$ by definition.
\end{proof}

\begin{lemma}
\label{lemma:hadbn}
  For every graph $G$,
$$\had_f(G) \leq \bn(G)\enspace.$$
\end{lemma}

\begin{proof}
  Let $\BB$ be a bramble in $G$ and let
  $w:\BB\rightarrow\mathbb{R}_{\geq 0}$ be a weight function, such
  that $\had_f(G) =\sum_{X\in\BB}w(X)$ and for each vertex $v$, the
  sum of the weights of the subgraphs in $\BB$ that contain $v$ is at
  most 1. Let $S$ be a hitting set for $\BB$.  Thus
$$|S|=\sum_{v\in S}1
\geq \sum_{v\in S}\sum_{X\in\BB:v\in X}w(X) = \sum_{X\in\BB}|X\cap
S|w(X) \geq \sum_{X\in\BB}w(X) =\had_f(G)\enspace.$$ That is, the
order of $\BB$ is at least $\had_f(G)$.  Hence $\bn(G) \geq
\had_f(G)$.
\end{proof}
Note this lemma is tight; consider $G = K_n$.

\citet{Wood-ProductMinor} proved that 
$\had(G\square K_2)\leq 2\tw(G)+2$ and \citet{ReedWood-EuJC} proved that $\glm(G) \leq 2\tw(G)+2$. More
precisely, Lemma~\ref{lemma:glmhad} and Lemma~\ref{lemma:hadbn} imply that
$$\glm(G) \leq \had(G\square K_2)\leq 2 \had_f(G) \leq 2\bn(G) = 2\tw(G)+2\enspace,$$
and for every integer $r\geq2$,
$$\glm(G) \leq 3 \had_r(G) \leq 3 \had_f(G) \leq 3\bn(G)
=3\tw(G)+3\enspace.$$ Conversely, \citet{ReedWood-EuJC} proved
that $$\tw(G) \leq c\, \glm(G)^4 \sqrt{ \log \glm(G) }\enspace,$$ for
some constant $c$.  Thus $\glm$, $\had(G\square K_2)$, $\had_f$,
$\had_r$ for each $r\geq2$, and $\tw$ are tied by polynomial
functions.

\section{Fractional Open Problems}
Given a graph $G$ define a \emph{$b$-fold colouring} for $G$ to be an assignment of $b$ colours to each vertex of $G$ such that if two vertices are adjacent, they have no colours assigned in common. We can consider this a generalisation of standard graph colouring, which is equivalent when $b=1$. A graph $G$ is \emph{$a\!:\!b$-colourable} when there is a $b$-fold colouring of $G$ with $a$ colours in total. Then define the \emph{$b$-fold chromatic number} $\chi_b(G) := \min\{a| G$ is $a\!:\!b$-colourable$\}$. So $\chi_1(G) = \chi(G)$. Then, define the \emph{fractional chromatic number} $\chi_f(G) = \lim_{b \rightarrow \infty} \frac{\chi_b(G)}{b}$.
See \citet{SU-FGT97} for an overview of the topic. \citet{ReedSeymour-JCTB98} proved that $\chi_f(G) \leq 2\,\had(G)$. Hence there is a relationship between the fractional chromatic number and Hadwiger's number. We have
$$\chi_f(G)\leq \chi(G)\quad\text{and}\quad\had(G)\leq\had_f(G)\leq\tw(G)+1\enspace.$$
 
Hadwiger's Conjecture asserts that $\chi(G)\leq\had(G)$, thus bridging the
gap in the above inequalities. 
Note that $\chi(G)\leq\tw(G)+1$. (Since $G$ has minimum degree
at most $\tw(G)$, a minimum-degree-greedy algorithm uses at most
$\tw(G)+1$ colours.) Thus the following two questions provide 
interesting weakenings of  Hadwiger's Conjecture:

\begin{conjecture}
\label{conjecture:p1}
$\chi(G)\leq \had_f(G)$.
\end{conjecture}

\begin{conjecture}
\label{conjecture:p2}
$\chi_f(G)\leq \had_f(G)$.
\end{conjecture}

Conjecture~\ref{conjecture:p2} was independently introduced in an equivalent form by \citet{Pedersen11}, along with a weaker form of Conjecture~\ref{conjecture:p1}. 

Finally, note that the above results prove that $\had_3$ is bounded by a
polynomial function of $\had_2$. Is $\had_3(G) \leq c\,\had_2(G)$ for
some constant $c$?

\section*{Acknowledgements}
Many thanks to Jacob Fox, Anders Pedersen and Bruce Reed for instructive conversations.

%%%%%%%%%%%%%%%%%%%%%%%%%%%%%%%%%%%%%%%%%%%%%%%%%%%%%%%%%%%%%%%%%%
\bibliographystyle{natbib}

\begin{thebibliography}{99}

\bibitem[Albertson {\em et~al.}(2004)Albertson, Chappell, Kierstead,
  K{\"u}ndgen, and Ramamurthi]{Albertson-EJC04}
Albertson, M.~O., Chappell, G.~G., Kierstead, H.~A., K{\"u}ndgen, A., and
  Ramamurthi, R. (2004).
\newblock Coloring with no 2-colored ${P}_4$'s.
\newblock {\em Electron. J. Combin.}, {\bf 11 \#R26}.

\bibitem[Arnborg {\em et~al.}(1987)Arnborg, Corneil, and
  Proskurowski]{arnborg-ktree}
Arnborg, S., Corneil, D.~G., and Proskurowski, A. (1987).
\newblock Complexity of finding embeddings in a {$k$}-tree.
\newblock {\em SIAM J. Algebraic Discrete Methods\/}, {\bf 8}(2), 277--284.

\bibitem[Bellenbaum and Diestel(2002)Bellenbaum and Diestel]{ShortDiestel}
Bellenbaum, P. and Diestel, R. (2002).
\newblock Two short proofs concerning tree-decompositions.
\newblock {\em Combin. Probab. Comput.}, {\bf 11}(6), 541--547.

\bibitem[Bodlaender(1993)Bodlaender]{Bodlaender-AC93}
Bodlaender, H.~L. (1993).
\newblock A tourist guide through treewidth.
\newblock {\em Acta Cybernet.}, {\bf 11}(1-2), 1--21.

\bibitem[Bodlaender(1998)Bodlaender]{Bodlaender-TCS98}
Bodlaender, H.~L. (1998).
\newblock A partial $k$-arboretum of graphs with bounded treewidth.
\newblock {\em Theoret. Comput. Sci.}, {\bf 209}(1-2), 1--45.

\bibitem[Bodlaender {\em et~al.}(2008)Bodlaender, Grigoriev, and
  Koster]{bodlaenderbramble}
Bodlaender, H.~L., Grigoriev, A., and Koster, A. M. C.~A. (2008).
\newblock Treewidth lower bounds with brambles.
\newblock {\em Algorithmica\/}, {\bf 51}(1), 81--98.

\bibitem[Chekuri and Chuzhoy(2014)Chekuri and Chuzhoy]{CC14}
Chekuri, C. and Chuzhoy, J. (2014).
\newblock Polynomial bounds for the grid-minor theorem.
\newblock In {\em Proc. 46th Annual ACM Symposium on Theory of Computing\/},
  STOC '14, pages 60--69. ACM.

\bibitem[Chuzhoy(2015)Chuzhoy]{Chuzhoy15}
Chuzhoy, J. (2015).
\newblock Excluded grid theorem: Improved and simplified.
\newblock In {\em Proc. 47th Annual ACM on Symposium on Theory of Computing\/},
  STOC '15, pages 645--654. ACM.

\bibitem[{Colin de Verdi\`ere}(1998){Colin de Verdi\`ere}]{yves}
{Colin de Verdi\`ere}, Y. (1998).
\newblock Multiplicities of eigenvalues and tree-width of graphs.
\newblock {\em J. Combin. Theory Ser. B\/}, {\bf 74}, 121--146.

\bibitem[Di~Giacomo {\em et~al.}(2012)Di~Giacomo, Didimo, Liotta, and
  Montecchiani]{DiGiacomo}
Di~Giacomo, E., Didimo, W., Liotta, G., and Montecchiani, F. (2012).
\newblock $h$-{Q}uasi planar drawings of bounded treewidth graphs in linear
  area.
\newblock In {\em Proceedings of the 38th International Conference on
  Graph-Theoretic Concepts in Computer Science\/}, WG'12, pages 91--102,
  Berlin. Springer-Verlag.

\bibitem[Diestel(2000)Diestel]{Diestel00a}
Diestel, R. (2000).
\newblock {\em Graph Theory\/}.
\newblock Springer.

\bibitem[Diestel {\em et~al.}(1999)Diestel, Jensen, Gorbunov, and
  Thomassen]{diestelHC}
Diestel, R., Jensen, T.~R., Gorbunov, K.~Y., and Thomassen, C. (1999).
\newblock Highly connected sets and the excluded grid theorem.
\newblock {\em J. Combin. Theory Ser. B\/}, {\bf 75}(1), 61--73.

\bibitem[Dujmovi{\'c} {\em et~al.}(2005)Dujmovi{\'c}, Morin, and Wood]{DMW05}
Dujmovi{\'c}, V., Morin, P., and Wood, D.~R. (2005).
\newblock Layout of graphs with bounded tree-width.
\newblock {\em SIAM J. Comput.}, {\bf 34}(3), 553--579.

\bibitem[Dvorak and Norin(2014)Dvorak and Norin]{balancedsep}
Dvorak, Z. and Norin, S. (2014).
\newblock Treewidth of graphs with balanced separations.
\newblock \arXiv{1408.3869}.

\bibitem[Flum and Grohe(2006)Flum and Grohe]{grohebook}
Flum, J. and Grohe, M. (2006).
\newblock {\em Parameterized Complexity Theory\/}.
\newblock Texts in Theoretical Computer Science. An EATCS Series. Springer.

\bibitem[Fox(2011)Fox]{Fox11}
Fox, J. (2011).
\newblock Constructing dense graphs with sublinear {H}adwiger number.
\newblock {\em J. Combin. Theory Ser. B \emph{(to appear)}\/}.
\newblock \arXiv{1108.4953}.

\bibitem[Fulkerson and Gross(1965)Fulkerson and Gross]{FG65}
Fulkerson, D.~R. and Gross, O.~A. (1965).
\newblock Incidence matrices and interval graphs.
\newblock {\em Pacific J. Math.}, {\bf 15}, 835--855.

\bibitem[Ganian {\em et~al.}(2014)Ganian, Hlin{\v{e}}n{\`y}, Langer,
  Obdr{\v{z}}{\'a}lek, Rossmanith, and Sikdar]{Ganian14}
Ganian, R., Hlin{\v{e}}n{\`y}, P., Langer, A., Obdr{\v{z}}{\'a}lek, J.,
  Rossmanith, P., and Sikdar, S. (2014).
\newblock Lower bounds on the complexity of {MSO1} model-checking.
\newblock {\em J. Comput. System Sciences\/}, {\bf 80}(1), 180--194.

\bibitem[Gavril(1974)Gavril]{gavril74}
Gavril, F. (1974).
\newblock The intersection graphs of subtrees in trees are exactly the chordal
  graphs.
\newblock {\em J. Combin. Theory Ser. B\/}, {\bf 16}, 47--56.

\bibitem[Grohe and Marx(2009)Grohe and Marx]{marxgrohe}
Grohe, M. and Marx, D. (2009).
\newblock On tree width, bramble size, and expansion.
\newblock {\em J. Combin. Theory Ser. B\/}, {\bf 99}(1), 218--228.

\bibitem[Hadwiger(1943)Hadwiger]{Hadwiger43}
Hadwiger, H. (1943).
\newblock \"{U}ber eine {K}lassifikation der {S}treckenkomplexe.
\newblock {\em Vierteljschr. Naturforsch. Ges. Z\"urich\/}, {\bf 88}, 133--142.

\bibitem[Halin(1976)Halin]{Halin76}
Halin, R. (1976).
\newblock {$S$}-functions for graphs.
\newblock {\em J. Geometry\/}, {\bf 8}(1-2), 171--186.

\bibitem[Harvey(2014)Harvey]{thesis}
Harvey, D.~J. (2014).
\newblock {\em On Treewidth and Graph Minors\/}.
\newblock Ph.D. thesis, The University of Melbourne.

\bibitem[Kawarabayashi and Kobayashi(2012)Kawarabayashi and
  Kobayashi]{KenandYusuke}
Kawarabayashi, K. and Kobayashi, Y. (2012).
\newblock Linear min-max relation between the treewidth of {$H$}-minor-free
  graphs and its largest grid minor.
\newblock In {\em 29th {I}nternational {S}ymposium on {T}heoretical {A}spects
  of {C}omputer {S}cience\/}, volume~14 of {\em LIPIcs. Leibniz Int. Proc.
  Inform.}, pages 278--289. Schloss Dagstuhl. Leibniz-Zent. Inform., Wadern.

\bibitem[Kawarabayashi and Mohar(2007)Kawarabayashi and Mohar]{KM-GC07}
Kawarabayashi, K. and Mohar, B. (2007).
\newblock Some recent progress and applications in graph minor theory.
\newblock {\em Graphs Combin.}, {\bf 23}(1), 1--46.

\bibitem[Kreutzer(2012)Kreutzer]{MR2904950}
Kreutzer, S. (2012).
\newblock On the parameterized intractability of monadic second-order logic.
\newblock {\em Log. Methods Comput. Sci.}, {\bf 8}(1), 1:27, 35.

\bibitem[Kreutzer and Tazari(2010a)Kreutzer and Tazari]{KT10}
Kreutzer, S. and Tazari, S. (2010a).
\newblock Lower bounds for the complexity of monadic second-order logic.
\newblock In {\em Proc. 25th Annual IEEE Symposium on Logic in Computer Science
  (LICS '10)\/}, pages 189--198. IEEE.

\bibitem[Kreutzer and Tazari(2010b)Kreutzer and Tazari]{MR2809681}
Kreutzer, S. and Tazari, S. (2010b).
\newblock On brambles, grid-like minors, and parameterized intractability of
  monadic second-order logic.
\newblock In {\em Proc. 21st {A}nnual {ACM}-{SIAM} {S}ymposium on {D}iscrete
  {A}lgorithms\/}, pages 354--364. SIAM.

\bibitem[K{\"u}ndgen and Pelsmajer(2008)K{\"u}ndgen and Pelsmajer]{KP-DM08}
K{\"u}ndgen, A. and Pelsmajer, M.~J. (2008).
\newblock Nonrepetitive colorings of graphs of bounded tree-width.
\newblock {\em Discrete Math.}, {\bf 308}(19), 4473--4478.

\bibitem[Leaf and Seymour(2015)Leaf and Seymour]{SeymourLeaf}
Leaf, A. and Seymour, P.~D. (2015).
\newblock Tree-width and planar minors.
\newblock {\em Journal of Combinatorial Theory, Series B\/}, {\bf 111}, 38--53.

\bibitem[Lipton and Tarjan(1979)Lipton and Tarjan]{LT79}
Lipton, R.~J. and Tarjan, R.~E. (1979).
\newblock A separator theorem for planar graphs.
\newblock {\em SIAM J. Appl. Math.}, {\bf 36}(2), 177--189.

\bibitem[Lucena(2007)Lucena]{lucenabramble}
Lucena, B. (2007).
\newblock Achievable sets, brambles, and sparse treewidth obstructions.
\newblock {\em Discrete Appl. Math.}, {\bf 155}(8), 1055--1065.

\bibitem[Markov and Shi(2011)Markov and Shi]{MS}
Markov, I. and Shi, Y. (2011).
\newblock Constant-degree graph expansions that preserve treewidth.
\newblock {\em Algorithmica\/}, {\bf 59}(4), 461--470.

\bibitem[Pedersen(2011)Pedersen]{Pedersen11}
Pedersen, A.~S. (2nd edition, 2011).
\newblock {\em Contributions to the Theory of Colourings, Graph Minors, and
  Independent Sets\/}.
\newblock Ph.D. thesis, Department of Mathematics and Computer Science,
  University of Southern Denmark.

\bibitem[Reed(1992)Reed]{Reedktreefinder}
Reed, B.~A. (1992).
\newblock Finding approximate separators and computing tree width quickly.
\newblock In {\em Proceedings of the twenty-fourth annual ACM Symposium on
  Theory of Computing\/}, STOC '92, pages 221--228, New York, USA. ACM.

\bibitem[Reed(1997)Reed]{Reed97}
Reed, B.~A. (1997).
\newblock Tree width and tangles: a new connectivity measure and some
  applications.
\newblock In {\em Surveys in Combinatorics\/}, volume 241 of {\em London Math.
  Soc. Lecture Note Ser.}, pages 87--162. Cambridge Univ. Press.

\bibitem[Reed and Seymour(1998)Reed and Seymour]{ReedSeymour-JCTB98}
Reed, B.~A. and Seymour, P.~D. (1998).
\newblock Fractional colouring and {H}adwiger's conjecture.
\newblock {\em J. Combin. Theory Ser. B\/}, {\bf 74}(2), 147--152.

\bibitem[Reed and Wood(2012)Reed and Wood]{ReedWood-EuJC}
Reed, B.~A. and Wood, D.~R. (2012).
\newblock Polynomial treewidth forces a large grid-like-minor.
\newblock {\em European J. Combin.}, {\bf 33}(3), 374--379.

\bibitem[Robertson and Seymour(983a)Robertson and Seymour]{RS-GraphMinors}
Robertson, N. and Seymour, P.~D. (1983--2012\noopsort{1983a}).
\newblock Graph minors {I--XXIII}.
\newblock {\em J. Combin. Theory Ser. B\/}.

\bibitem[Robertson and Seymour(1986a)Robertson and
  Seymour]{RS-GraphMinorsII-JAlg86}
Robertson, N. and Seymour, P.~D. (1986a).
\newblock Graph minors. {II}. {A}lgorithmic aspects of tree-width.
\newblock {\em J. Algorithms\/}, {\bf 7}(3), 309--322.

\bibitem[Robertson and Seymour(1986b)Robertson and
  Seymour]{RS-GraphMinorsV-JCTB86}
Robertson, N. and Seymour, P.~D. (1986b).
\newblock Graph minors. {V}. {E}xcluding a planar graph.
\newblock {\em J. Combin. Theory Ser. B\/}, {\bf 41}(1), 92--114.

\bibitem[Robertson and Seymour(1990)Robertson and
  Seymour]{RS-GraphMinorsIV-JCTB90}
Robertson, N. and Seymour, P.~D. (1990).
\newblock Graph minors. {IV}.\ {T}ree-width and well-quasi-ordering.
\newblock {\em J. Combin. Theory Ser. B\/}, {\bf 48}(2), 227--254.

\bibitem[Robertson and Seymour(1991)Robertson and
  Seymour]{RS-GraphMinorsX-JCTB91}
Robertson, N. and Seymour, P.~D. (1991).
\newblock Graph minors. {X}. {O}bstructions to tree-decomposition.
\newblock {\em J. Combin. Theory Ser. B\/}, {\bf 52}(2), 153--190.

\bibitem[Robertson and Seymour(1995)Robertson and
  Seymour]{RS-GraphMinorsXIII-JCTB95}
Robertson, N. and Seymour, P.~D. (1995).
\newblock Graph minors. {XIII}. {T}he disjoint paths problem.
\newblock {\em J. Combin. Theory Ser. B\/}, {\bf 63}(1), 65--110.

\bibitem[Robertson {\em et~al.}(1993)Robertson, Seymour, and
  Thomas]{RST-Comb93}
Robertson, N., Seymour, P.~D., and Thomas, R. (1993).
\newblock Hadwiger's conjecture for ${K}\sb 6$-free graphs.
\newblock {\em Combinatorica\/}, {\bf 13}(3), 279--361.

\bibitem[Robertson {\em et~al.}(1994)Robertson, Seymour, and
  Thomas]{RST-JCTB94}
Robertson, N., Seymour, P.~D., and Thomas, R. (1994).
\newblock Quickly excluding a planar graph.
\newblock {\em J. Combin. Theory Ser. B\/}, {\bf 62}(2), 323--348.

\bibitem[Rose(1974)Rose]{rose-ktree}
Rose, D.~J. (1974).
\newblock On simple characterizations of {$k$}-trees.
\newblock {\em Discrete Math.}, {\bf 7}, 317--322.

\bibitem[Scheffler(1989)Scheffler]{Scheffler}
Scheffler, P. (1989).
\newblock {\em Die Baumweite von Graphen als ein Ma\ss f\"{u}r die
  Kompliziertheit algorithmischer Probleme\/}.
\newblock Ph.D. thesis, Akademie der Wissenschaften der DDR, Berlin.

\bibitem[Scheinerman and Ullman(1997)Scheinerman and Ullman]{SU-FGT97}
Scheinerman, E.~R. and Ullman, D.~H. (1997).
\newblock {\em Fractional graph theory\/}.
\newblock Wiley.

\bibitem[Seymour and Thomas(1993)Seymour and Thomas]{SeymourThomas-JCTB93}
Seymour, P.~D. and Thomas, R. (1993).
\newblock Graph searching and a min-max theorem for tree-width.
\newblock {\em J. Combin. Theory Ser. B\/}, {\bf 58}(1), 22--33.

\bibitem[van~der Holst(1996)van~der Holst]{holst}
van~der Holst, H. (1996).
\newblock {\em Topological and Spectral Graph Characterizations\/}.
\newblock Ph.D. thesis, Amsterdam University, Netherlands.

\bibitem[van Leeuwen(1990)van Leeuwen]{jvanleeuwen}
van Leeuwen, J. (1990).
\newblock Graph algorithms.
\newblock In J.~van Leeuwen, editor, {\em Handbook of Theoretical Computer
  Science, Vol 1: Algorithms and Complexity\/}, chapter~10, pages 525--631.
  Elsevier Science Publishers, Amsterdam.

\bibitem[Wood(2011)Wood]{Wood-ProductMinor}
Wood, D.~R. (2011).
\newblock Clique minors in {C}artesian products of graphs.
\newblock {\em New York J. Math.}, {\bf 17}, 627--682.

\end{thebibliography}
\def\soft#1{\leavevmode\setbox0=\hbox{h}\dimen7=\ht0\advance \dimen7
  by-1ex\relax\if t#1\relax\rlap{\raise.6\dimen7
  \hbox{\kern.3ex\char'47}}#1\relax\else\if T#1\relax
  \rlap{\raise.5\dimen7\hbox{\kern1.3ex\char'47}}#1\relax \else\if
  d#1\relax\rlap{\raise.5\dimen7\hbox{\kern.9ex \char'47}}#1\relax\else\if
  D#1\relax\rlap{\raise.5\dimen7 \hbox{\kern1.4ex\char'47}}#1\relax\else\if
  l#1\relax \rlap{\raise.5\dimen7\hbox{\kern.4ex\char'47}}#1\relax \else\if
  L#1\relax\rlap{\raise.5\dimen7\hbox{\kern.7ex
  \char'47}}#1\relax\else\message{accent \string\soft \space #1 not
  defined!}#1\relax\fi\fi\fi\fi\fi\fi}\newcommand{\noopsort}[1]{}

%%%%%%%%%%%%%%%%%%%%%%%%%%%%%%%%%%%%%%%%%%%%%%%%%%%%%%%%%%%%%%%%%%

\end{document}